\documentclass[12pt,leqno,a4paper]{amsart}
\usepackage{amssymb,enumerate}
\newcommand{\FF}{{\mathbb{F}}}

\newcommand{\QQ}{{\mathbb{Q}}}
\newcommand{\ZZ}{{\mathbb{Z}}}
\newcommand{\CC}{{\mathbb{C}}}

\newcommand{\fS}{{\mathfrak{S}}}

\newcommand{\bI}{{\mathbf{I}}}

\newcommand{\cH}{{\mathcal{H}}}

\newcommand{\Ind}{{\operatorname{Ind}}}
\newcommand{\Irr}{{\operatorname{Irr}}}
\newcommand{\Uch}{{\operatorname{Uch}}}
\newcommand{\id}{{\operatorname{id}}}

\newcommand{\CF}{{\operatorname{CF}}}
\newcommand{\Fr}{{\operatorname{Fr}}}

\newcommand{\Aut}{{\operatorname{Aut}}}

\newcommand{\GL}{{\operatorname{GL}}}
\newcommand{\GU}{{\operatorname{GU}}}

\newcommand{\tw}[1]{{}^#1\!}
\newcommand{\thi}{{\tilde\phi}}
\newcommand{\teps}{\tilde\epsilon}

\let\al=\alpha
\let\bt=\beta
\let\la=\lambda
\let\sg=\sigma
\let\eps=\epsilon

\newtheorem{thm}{Theorem}[section]
\newtheorem{lem}[thm]{Lemma}

\newtheorem{cor}[thm]{Corollary}
\newtheorem{prop}[thm]{Proposition}

\theoremstyle{definition}
\newtheorem{exmp}[thm]{Example}

\theoremstyle{remark}
\newtheorem{rem}[thm]{Remark}

\begin{document}

\title{Frobenius--Schur indicators of unipotent characters and the
twisted involution module}

\date{\today}

\author{Meinolf Geck}
\address{Institute of Mathematics, Aberdeen University,
  Aberdeen AB24 3UE, Scotland, UK.}
\email{m.geck@abdn.ac.uk}
\author{Gunter Malle}
\address{FB Mathematik, TU Kaiserslautern, Postfach 3049,
         67653 Kaisers\-lautern, Germany.}
\email{malle@mathematik.uni-kl.de}


\begin{abstract}
Let $W$ be a finite Weyl group and $\sg$ be a non-trivial graph automorphism
of $W$. We show a remarkable relation between the $\sg$-twisted involution
module for $W$ and the Frobenius--Schur indicators of the unipotent 
characters of a corresponding twisted finite group of Lie type. This 
extends earlier results of Lusztig--Vogan for the untwisted case and 
then allows us to state a general result valid for any finite group of 
Lie type. Inspired by recent work of Marberg, we also formally define 
Frobenius--Schur indicators for ``unipotent characters'' of twisted 
dihedral groups.
\end{abstract}

\maketitle

\pagestyle{myheadings}
\markboth{Geck--Malle}{Frobenius--Schur indicators and the twisted involution
module}

\section{Introduction} \label{sec:intro}

Let $G$ be a connected reductive algebraic group over $\overline{\FF}_p$
(where $p$ is a prime) and $F \colon G \rightarrow G$ be an endomorphism
such that some power of $F$ defines a rational structure on $G$ relative
to a finite subfield of $\overline{\FF}_p$. Then the fixed point set
$G^F$ is a finite group of Lie type. Let $W$ be the Weyl group of $G$ 
and $S$ be a corresponding set of simple reflections, defined with respect 
to an $F$-stable maximal torus contained in an $F$-stable Borel subgroup of
$G$. Let $\Uch(G^F)$ be the set of unipotent characters, as defined by 
Deligne and Lusztig \cite{DeLu}. An irreducible character $\chi$ of $G^F$ is
unipotent if and only if $\chi$ appears with non-zero multiplicity in some 
Deligne--Lusztig virtual character $R_{T_w,1}$ where $T_w$ is an $F$-stable 
maximal torus ``of type $w$'' and $1$ stands for the trivial character of 
$T_w^F$. 

Assuming that $G$ is of split type, Lusztig and Vogan \cite{LV} have 
established a connection between the Frobenius--Schur indicators of the 
unipotent characters of $G^F$ and a certain involution module for $W$
which appeared in the work of Kottwitz \cite{Ko}. More precisely, let 
$\CF_0(G^F)$ be the space of all \emph{uniform} unipotent class functions on
$G^F$, that is, the subspace of the space of class functions on $G^F$ which is 
spanned by all the virtual characters $R_{T_w,1}$. For any class function
$\psi$ on $G^F$, denote by $\psi_0$ the orthogonal projection onto the
subspace $\CF_0(G^F)$. Then, by \cite[6.4]{LV}, we have the following
remarkable identity (assuming that $G$ is of split type):
\[ \bigl(\sum_{\chi \in \Uch(G^F)} \nu(\chi)\, \chi\bigr)_0=\frac{1}{|W|} 
   \sum_{w \in W} \rho(w) R_{T_w,1};\]
here, $\nu(\chi)$ denotes the Frobenius--Schur indicator of $\chi$ and
$\rho$ is the character of Kottwitz' involution module of $W$.
In this paper, we extend this to the situation where $W$ is not 
necessarily of split type. Then $F$ induces a non-trivial automorphism 
$\sg \colon W \rightarrow W$ such that $\sg(S)=S$. We shall consider a 
canonical extension of the involution module to the semidirect product 
$\tilde{W}=W \rtimes\langle\sg \rangle$; let $\tilde{\rho}$ be the 
character of this extended module. Then we shall show that the above 
identity remains true where, on the right hand side, the term $\rho(w)$ 
needs to be replaced by $\tilde{\rho}(w\sg)$. With this modification, we 
actually obtain a statement valid for any $G^F$; see Corollary~\ref{cor1}. 

The proof of this result proceeds as in \cite{LV}, by reformulating
the desired identity in terms of Lusztig's Fourier transform 
\cite[Chap.~4]{LuB} (where we use the Fourier matrices in \cite{GM}
for the Suzuki and Ree groups). Our main result, Theorem~\ref{thm:main}, 
shows that as in the split case, the multiplicities in the decomposition of 
$\tilde{\rho}$ as a class function on the coset $W.\sg$ are obtained via 
Fourier transform from the Frobenius--Schur indicators of the unipotent 
characters of $G^F$. We prove Theorem~\ref{thm:main} by explicitly 
decomposing $\tilde{\rho}$ in all relevant cases (see 
Section~\ref{sec:andn} for classical types and Section~\ref{sec:exc} 
for exceptional types) and then comparing with the Fourier transform 
of the vector of Frobenius--Schur indicators in Section~\ref{sec:FS}. 
The main difficulty is of a somewhat technical nature: in the twisted case, 
both the Fourier matrices and the formulae for the decomposition of 
$\tilde{\rho}$ essentially rely on choices of extensions of $\sg$-stable 
irreducible characters of $W$ and, a priori, there is no reason why the 
choices on both sides should fit together. We shall see that the formulae
for the decomposition of $\tilde{\rho}$ become particularly simple when we
choose Lusztig's \emph{preferred extensions}, as defined in \cite[17.2]{cs4}.
Note also that in the split case all Frobenius--Schur indicators of unipotent
characters are $0$ or $1$, which is no longer the case in our setting. 
Already in type $\tw2A_n$, it is quite remarkable how the multiplicities in 
the decomposition of $\tilde{\rho}$ match the distribution of the $\pm 1$
values of the Frobenius--Schur indicators of the unipotent characters
of $\GU_n(q)$.

Finally, the construction of $\tilde{\rho}$ also works for the case where $W$
is any dihedral group and $\sg\neq \id$. Thus, inspired by Marberg
\cite{Mar}, our computations also allow us to formally define 
Frobenius--Schur indicators for ``unipotent characters'' of twisted 
dihedral groups; see Theorem~\ref{thm:cox}.

\section{The extended involution module } \label{sec:invmod}

Let $(W,S)$ be a finite Coxeter group, with distinguished set of generators
$S$. We write $\bI:=\{w\in W\mid w^2=1\}$ for the subset of elements of order
at most~2. Let $\cH(W)$ be the one-parameter Iwahori--Hecke algebra over
$\QQ[v^{\pm1}]$ attached to $W$ with parameter $v^2$, with standard basis 
$\{T_w\mid w\in W\}$. (This basis is normalised such that the quadratic 
relations read $(T_s-v^2)(T_s+1)=0$ for $s \in S$.) In \cite{Lu12}, Lusztig 
shows that an action of $\cH(W)$ on a free $\QQ[v^{\pm 1}]$-module with 
basis $\{a_w\mid w\in \bI\}$ can be defined by the following formulae:
$$T_s.a_w=\begin{cases} v\,a_w+(v+1)a_{sw}& \text{if }sw=ws,\ l(ws)>l(w),\\
     (v^2-v-1)a_w+(v^2-v)a_{sw}& \text{if }sw=ws,\ l(ws)<l(w),\\
     a_{sws}& \text{if }sw\ne ws,\ l(ws)>l(w),\\
     (v^2-1)a_w+v^2\,a_{sws}& \text{if }sw\ne ws,\ l(ws)<l(w),
     \end{cases}$$
for $s\in S$, $w\in \bI$. We write $R_v$ for this representation of
$\cH(W)$. As in \cite[6.3]{LV}, this representation induces a representation
$R$ of $W$ on the $\QQ$-vector space $V=\bigoplus_{w\in \bI}\QQ a_w$,
defined by
$$R(s).a_w:=\begin{cases} -a_w& \text{if }sw=ws,\ l(ws)<l(w),\\
     a_{sws}& \text{else}. \end{cases}$$
Let $C$ be a conjugacy class of $W$ contained in $\bI$. Then it is
clear that the subspace $V_C:=\langle a_w \mid w \in C\rangle\subseteq
V$ is a submodule; furthermore, we have a direct sum decomposition
$V=\bigoplus_C V_C$ where $C$ runs over the conjugacy classes of
$W$ contained in $\bI$. Now let us fix such a conjugacy class $C$. Then
an alternative description of $V_C$ is given as follows.

By \cite[3.2.10]{GePf}, $C$ contains an element of the form $w_I$
where $I \subseteq S$ and $w_I$ is the longest element in the parabolic
subgroup $W_I\subseteq W$; furthermore, $w_I$ is central in $W_I$.
By \cite[2.1.15]{GePf}, we have a semidirect product decomposition
$N_W(W_I)=Y\ltimes W_I$ where $Y$ is a group consisting of certain
distinguished left coset representatives of $W_I$ in $W$ such that
$I$ is invariant under conjugation with all $y \in Y$; in particular,
we have $l(yw')=l(y)+l(w')$ for all $y \in Y$ and $w' \in W_I$. Let
$\varepsilon_I \colon N_W(W_I) \rightarrow \{\pm 1\}$ be the trivial
extension of the sign representation of $W_I$, that is, we have
\[ \varepsilon_I(yw')=(-1)^{l(w')} \qquad \mbox{for all $y \in Y$ and
   $w' \in W_I$}.\]

\begin{lem}   \label{luvo1}
 In the above setting, we have
  \[ N_W(W_I)=C_W(w_I) \qquad \mbox{and}\qquad V_C\cong
  \Ind_{C_W(w_I)}^W(\varepsilon_I).\]
\end{lem}

\begin{proof}
First we show the equality $N_W(W_I)=C_W(w_I)$.
Let $w \in N_W(W_I)$. We write $w=yw'$ where $y \in Y$ and $w' \in W_I$.
Since $yIy^{-1}=I$, we have $yw_Iy^{-1}=w_I$. Since $w_I$ is central in
$W_I$, we conclude that $w=yw' \in C_W(w_I)$. Conversely, let $w \in
C_W(w_I)$. We write $w=xw'$ where $w' \in W_I$ and $x$ is some
distinguished left coset representative of $W_I$ in $W$. Since $w'$
commutes with $w_I$, we also have $xw_I=w_Ix$. Since $l(xw_I)=l(x)+l(w_I)$,
we conclude that $l(x^{-1}w_I)=l(w_Ix)=l(x)+l(w_I)$ and so both $x$ and
$x^{-1}$ are distinguished left coset representatives. Hence, by
\cite[2.1.12]{GePf}, we have $xW_Ix^{-1} \cap W_I=W_J$ where $J= xIx^{-1}
\cap I$. But $w_I=xw_Ix^{-1}$ lies in this intersection and, hence, in
$W_J$. It follows that $I=J$ and so $x \in N_W(W_I)$.

To prove the statement concerning $V_C$, it is sufficient to show that
\begin{equation*}
  (yw').a_{w_I}=(-1)^{l(w')}a_{w_I} \qquad \text{for all $y \in Y$ and
  $w'\in W_I$}.\tag{$*$}
\end{equation*}
This is seen as follows. For any $s \in I$, we have $s.a_{w_I}=-a_{w_I}$.
Consequently, we have $w'.a_w=(-1)^{l(w')}a_{w_I}$. Thus, it remains to
show that $y.a_{w_I}=a_{w_I}$ for all $y \in Y$. We shall in fact show
that $x.a_{w_I}=a_{xw_Ix^{-1}}$ where $x$ is any distinguished left coset
representative of $W_I$ in $W$. We proceed by induction on $l(x)$. If
$x=1$, the assertion is clear. Now assume that $x\neq 1$ and choose
$s \in S$ such that $l(sx)<l(x)$. By Deodhar's Lemma \cite[2.1.2]{GePf},
we also have that $z:=sx$ is a distinguished left coset representative.
Hence, using induction, we have $z.a_{w_I}=a_{zw_Iz^{-1}}$ and so
\[ x.a_{w_I}=s.a_{zw_Iz^{-1}}.\]
Given the formula for the action of a generator on the basis elements
of $V_C$, it now suffices to show that $s$ does not commute with
$zw_Iz^{-1}$ or that $l(zw_Iz^{-1}s)>l(zw_Iz^{-1})$. Assume, if possible,
that none of these two conditions is satisfied, that is, we have that $s$
commutes with $zw_Iz^{-1}$ and $l(szw_Iz^{-1})=l(zw_Iz^{-1}s)<l(zw_Iz^{-1})$.
Then the ``Exchange Lemma'' (see \cite[Exc.~1.6]{GePf})  and the fact that
$l(szw_I)=l(z)+l(w_I)+1$ imply that $szw_Iz^{-1}=zw_Iu$ where $l(u)<l(z)$.
Since $s$ commutes with $zw_Iz^{-1}$, we have $zw_Iz^{-1}s=szw_Iz^{-1}=
zw_Iu$ and so $z^{-1}s=u$. This would imply that $l(u)=l(z^{-1}s)=
l(sz)>l(z)$, a contradiction. Hence, the assumption was wrong and so
$x.a_{w_I}=a_{xw_Ix^{-1}}$, as required.
\end{proof}

\begin{rem}   \label{luvo1a}
It is stated in \cite[6.3]{LV} that the
$W$-module $V$ is isomorphic to the involution module constructed by
Kottwitz \cite{Ko}. Lemma~\ref{luvo1} provides an easy proof of this
statement. Indeed, assume that $W$ is realised as a subgroup generated by
reflections in $\GL(E)$, where $E$ is a finite-dimensional Euclidean
space. Let $\Phi$ be the corresponding root system and $\Phi^+$
be the set of positive roots defined by $S$. Let $t \in \bI$.
Following Kottwitz \cite{Ko}, a root $\alpha \in \Phi$ is called imaginary
with respect to $t$ if $t(\alpha)= -\alpha$. Then Kottwitz
defines a linear representation $\delta \colon C_W(t) \rightarrow
\{\pm 1\}$ such that $\delta(w)=(-1)^k$ where $k$ is the number of positive
imaginary roots which are sent to negative roots by~$w$. Now, if $t=
w_I$ (where $w_I$ is central in $W_I$ as above), then one easily sees that
the imaginary roots with respect to $w_I$ are precisely the roots in the
parabolic subsystem $\Phi_I \subseteq \Phi$ corresponding to $W_I$.
Furthermore, let $w \in C_W(w_I)$ and write $w=yw'$ where $y \in Y$
and $w'\in W_I$. Then $l(w')$ is the number of positive roots in $\Phi_I$
which are sent to negative roots by $w'$. Since $y$ sends all positive
roots in $\Phi_I$ to positive roots, we conclude that $l(w')$ is the number
of positive imaginary roots sent to negative roots by $w$. Thus, we have
$\delta=\varepsilon_I$, which yields the isomorphism between $V$ and
the involution module constructed by Kottwitz.
\end{rem}

Now let $\sg\in\Aut(W)$ be an automorphism of $W$ as a Coxeter group,
that is, $\sg$ stabilizes the set $S$ of distinguished generators. In
particular, $\sg$ preserves lengths in $W$.
Then there is a natural action of $\sg$ on the Iwahori--Hecke algebra
$\cH(W)$, and we obtain the extended Hecke algebra $\tilde\cH(W)$, the
associative $\QQ[v^{\pm1}]$-algebra generated by $\cH(W)$ together with an
additional element $T_\sg$ subject to the relations $T_\sg^d=1$ and
$T_\sg T_s =T_{\sg(s)}T_\sg$ for all $s\in S$, where $d=o(\sg)$
denotes the order of $\sg$. With this we have the following\footnote{This
result also appears in the first version of a recent preprint \cite{L12}
of Lusztig.}:

\begin{prop}   \label{prop:rho1}
 The above representation $R_v$ of $\cH(W)$ extends to a representation
 $\tilde R_v$ of $\tilde\cH(W)$ via
 $$T_\sg.a_w:=a_{\sg(w)}\qquad\text{for all }w\in \bI.$$
\end{prop}

\begin{proof}
It suffices to check that the additional relations of $\tilde\cH(W)$ involving
$T_\sg$ are respected. It is clear that $T_\sg^d$ acts as the identity.
Furthermore, for $s\in S$ with $\sg(s)=s'$ we have
$$T_\sg T_s.a_w=\left.\begin{cases}
     v\,a_{\sg(w)}+(v+1)a_{s'\sg(w)}& \text{if }sw=ws,\ l(ws)>l(w)\\
     (v^2-v-1)a_{\sg(w)}+(v^2-v)a_{s'\sg(w)}& \text{if }sw=ws,\ l(ws)<l(w)\\
     a_{s'\sg(w)s'}& \text{if }sw\ne ws,\ l(ws)>l(w)\\
     (v^2-1)a_{\sg(w)}+v^2\,a_{s'\sg(w)s'}& \text{if }sw\ne ws,\ l(ws)<l(w)
     \end{cases}\right\}=T_{s'}T_\sg.a_w$$
for all $w\in \bI$, since $\sg$ leaves the length function invariant.
\end{proof}

Let $\tilde{W}$ be the semidirect product of $W$ with $\langle \sg \rangle$.
Thus, $\tilde{W}=\langle W,\sg\rangle$ and, in $\tilde{W}$, we
have the identity $\sg(w)=\sg w \sg^{-1}$ for all $w \in W$. By an argument
similar to that in the proof of Proposition~\ref{prop:rho1}, we obtain:

\begin{prop}   \label{prop:rho2}
 The representation $R$ of $W$ extends to a representation
 $\tilde R$ of $\tilde{W}=\langle W,\sg\rangle$ via
 $$\tilde R(\sg).a_w:=a_{\sg(w)}\qquad\text{for all }w\in \bI.$$
\end{prop}

Let us write $\rho$ for the character of the representation $R$, and similarly
$\tilde\rho$ for that of $\tilde R$. We are interested in the decomposition of
$\tilde{\rho}$ as a class function on the coset $W.\sg$. In the case that
$\sg=\id$, this decomposition was obtained by Kottwitz \cite{Ko}, Casselman
\cite{Cas} for Weyl groups and by Marberg \cite{Mar} for the
non-crystallographic Coxeter groups. Here, we discuss the case when $\sg\ne\id$.

\begin{rem}   \label{firstrem}
Assume that $C\subseteq\bI$ is a $\sg$-stable
conjugacy class of $W$. Then $V_C$ is a $\tilde W$-submodule of $V$ and
$$\tilde\rho_C:=\tilde\rho|_{V_C}
  =\Ind_{C_{\tilde W}(t)}^{\tilde W}(\teps_C)\qquad\text{for any }t\in C,$$
where $\teps_C:C_{\tilde W}(t)\rightarrow\{\pm1\}$ is the linear character
defined by $\tilde R(w).a_t=\teps_C(w)\,a_t$ for $w\in C_{\tilde W}(t)$.
On the other hand, if $C\subseteq \bI$ is a conjugacy class which is not
$\sg$-invariant, then the trace of any element in $W.\sg$ on
$\sum_{i=1}^d V_{\sg^i(C)}$ is zero. Thus, as a class function on
$W.\sg$, we have
\[ \tilde{\rho}=\sum_{C} \tilde{\rho}_C,\]
where the sum runs  over all $\sg$-stable conjugacy classes $C$ of
$W$ which are contained in $\bI$. Now let $C$ be such a conjugacy class.
By Frobenius reciprocity, we have
\[ \langle \tilde{\rho}_C,\Ind_W^{\tilde{W}}(\phi)\rangle=
   \langle \rho_C,\phi\rangle \qquad \mbox{for all $\phi \in \Irr(W)$}.\]
Assume further that $d=o(\sg)$ is a prime number (which will be the case
in all situations we shall consider). If $\phi^\sg \neq \phi$,
then $\Ind_{W}^{\tilde{W}}(\phi)$ is irreducible; otherwise,
$\Ind_{W}^{\tilde{W}}(\phi)$ is the sum of $d$ distinct extensions of
$\phi$ to $\tilde{W}$. Hence, as a class function on $W.\sg$, we have
a unique decomposition
\[ \tilde{\rho}_C=\sum_{\phi \in \Irr(W)^\sg} n_\phi
\, \tilde{\phi} \qquad (n_\phi \in {\ZZ}[\sqrt[d]{1}]),\]
where each $\tilde{\phi}$ is a fixed extension of $\phi \in \Irr(W)^\sg$
to the coset $W.\sg$. So the problem will be to fix choices for the
extensions $\thi$ and to determine the corresponding coefficients $n_\phi$.
\end{rem}

In order to deal with the various possibilities for $W,\sg$ where $W$ is
irreducible, the following general result will be useful. Assume that
$\sg\ne\id$ acts as an inner automorphism of $W$. Since the only non-trivial
element of $W$ fixing $S$ is the longest element $w_0$ of $W$, this implies
that $\sg$ acts by conjugation with $w_0$. By the classification of finite
Coxeter groups, $W$ is of type $A_n$, $D_{2n+1}$, $I_2(2m+1)$ or $E_6$.
Note that in all of these groups, the centralizer of $\sg$ contains a Sylow
$2$-subgroup of $W$, so fixes some element in any conjugacy class
$C\subseteq \bI$. Note also that, in these cases, the element $w_0\sg$ 
lies in the center of $\tilde{W}$ and so $\thi(w_0\sg)=\pm \phi(1)$ for
every extension $\thi$ of $\phi \in \Irr(W)^\sg$.

\begin{prop}   \label{prop:inn}
 Assume that $\sg$ acts on $W$ by conjugation with $w_0$. Let $C\subseteq
 \bI$ be a $\sg$-stable conjugacy class. Then, as a class function on
 $W.\sg$, we have
 \[ \tilde{\rho}_C=\sum_{\phi \in \Irr(W)} \langle \rho_C,\phi\rangle\,
 \thi,\]
 where $\thi$ denotes the (unique) extension of $\phi$ to
 $W.\sg$ such that $(-1)^{l(t)}\,\thi(w_0\sg)>0$, for any $t\in C$.
 In particular, $\phi(w_0)$ has the same sign for all constituents $\phi$
 of $\rho_C$ with $\phi(w_0)\ne0$.
\end{prop}

\begin{proof}
By assumption, $w_0\sg$ lies in the center of $\tilde W$, so it acts by the
scalar $\eps_C(w_0\sg)$ in the representation $\tilde R_C$. But $w_0$ sends
all positive roots to negative ones while $\sg$ stabilizes the set of positive
roots, so by the explicit formula for $\eps_C$ above we have that
$\eps_C(w_0\sg)=(-1)^{l(t)}$. In particular,
$\psi(w_0\sg)=(-1)^{l(t)}\psi(1)\ne0$ for any constituent $\psi\in
\Irr(\tilde{W})$ of $\tilde\rho_C$, whence $\langle \rho_C,\phi\rangle=
\langle \tilde\rho_C,\thi\rangle$ for the extension $\thi$ of $\phi$ given
in the statement. Since $\tilde R_C(\sg)$ is
a permutation matrix, it is then clear that $\tilde\rho_C(w_0)=\rho_C(w_0)$
has the same sign as $\tilde\rho_C(w_0\sg)$. Thus for all constituents $\phi$
of $\rho_C$ with $\phi(w_0)\ne0$, $\phi(w_0)$ has the same sign $(-1)^{l(t)}$.
\end{proof}

\begin{rem}   \label{rempref}
Let $\phi\in \Irr(W)$ be such that $\phi^\sg= \phi$. In \cite[17.2]{cs4}, 
Lusztig has defined the notion of a \emph{preferred extension} $\thi$ of 
$\phi$ to $\tilde W$. In the cases where $W$ is an irreducible Weyl group
and $\sg$ is ``ordinary'' in the sense of \cite[3.1]{LuB} (that is, 
whenever $s\neq s'$ in $S$ are in the same $\sg$-orbit, then the product 
$ss'$ has order $2$ or $3$), these are given as follows.
\begin{itemize}
\item If $\sg=\id$, then $\thi=\phi$.
\item If $\sg$ acts by conjugation with $w_0$ and $W$ is of type $A_n$ 
($n \geq 2$) or $E_6$, then  $\thi$ is the unique extension such that
$\thi(w_0\sg)=(-1)^{a_\phi}\phi(1)$ where $a_\phi$ is the invariant 
defined in \cite[4.1]{LuB} (in terms of the generic degree polynomial
associated with $\phi$).
\item If $\sg$ has order $3$ and $W$ is of $D_4$, then $\thi$ is the
unique extension which is defined over $\QQ$ (see \cite[3.2]{LuB}).
\item If $\sg$ has order $2$ and $W$ is of type $D_n$ ($n \geq 4$), then
$\thi$ is defined as follows. In this case, $\tilde{W}$ can be identified
with a Weyl group of type $B_n$ and, as in \cite[4.18]{LuB}, the 
irreducible characters of $\tilde{W}$ which remain irreducible upon
restriction to $W$ are parametrised by certain symbols with two rows 
of equal length (an upper row and a lower row). Then $\thi \in 
\Irr(\tilde{W})$ is the {\em preferred extension} of $\phi$ if the
corresponding symbol has the following property: the smallest entry which
appears in only one row appears in the lower row.
\end{itemize}
\end{rem}

\section{Types $A_n$ and $D_n$} \label{sec:andn}

In this and the following section we explicitly determine the
decomposition of $\tilde{\rho}$ (as a class function on $W.\sg$, see
Remark~\ref{firstrem}), in all cases where $W$ is irreducible and
$\sg \neq \id$.

\subsection{} We begin with the case where $W=\fS_n$. Here, the irreducible
characters
of $W$ are labelled by partitions $\al\vdash n$, and we write $\phi_\al$
for the corresponding character. The non-trivial $\sg$ is given by
conjugation with the longest element $w_0$.

\begin{prop}   \label{prop:signA}
 Let $W=\fS_n$ and $\sg$ the graph automorphism of order~$2$. As a
 class function on $W.\sg$, we have
 $$\tilde\rho=\sum_{\al\vdash n}\thi_\al,$$
 where $\thi_\al$ denotes Lusztig's preferred extension as in 
 Remark~\ref{rempref}.
\end{prop}

\begin{proof}
Let $\al\vdash n$ and denote by $m_\al$ the number of odd parts of the 
conjugate partition $\al'$. By \cite[3.1]{Ko}, $\phi_\al$ is a constituent 
of $\rho_C$, where $C$ is the class of involutions of $\fS_n$ with exactly 
$m_\al$ fixed points. Now clearly there is such an involution $t$ of
length $(n-m_\al)/2$, whence $\eps_C(w_0\sg)=(-1)^{l(t)}=(-1)^{(n-m_\al)/2}$.
Thus, by Proposition~\ref{prop:inn}, we have $\tilde\rho=\sum_{\al\vdash n}
\thi_\al$ where $\thi_\al$ is the unique extension such that $\thi_\al(w_0
\sg)=(-1)^{(n-m_\al)/2}\phi_\al(1)$. Finally, let $\al \vdash n$ and denote 
by $(\al_1', \ldots,\al_r')$ the parts of the conjugate partition $\al'$. 
Then by \cite[5.4.1]{GePf} we have
\[2a_\chi\equiv \sum_{1\leq i \leq r}\al_i'(\al_i'-1)=\sum_{1 \leq i \leq r}
\al_i'^2-n\equiv m_\al-n\pmod4. \]
This shows that $(n-m_\al)/2 \equiv a_\al \pmod 4$ and so the extension
$\thi$ chosen above is Lusztig's preferred extension. 
\end{proof}

\subsection{}
For the remainder of this section, let $W$ be of type $D_n$ ($n \geq 2$)
and $\sg$ a graph automorphism of order~$2$. In this case, the notation can
be arranged such that $\tilde{W}=\langle W,\sg \rangle$ is a Coxeter group 
of type $B_n$ with generating set $S=\{\sigma,s_1, \ldots,s_{n-1}\}$ and 
diagram:
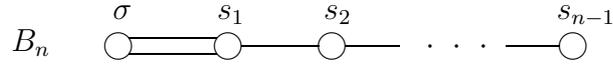
\begin{center}
\begin{picture}(250,22)
\put(  0,  3){$B_n$}
\put( 40,  5){\circle{10}}
\put( 44,  2){\line(1,0){33}}
\put( 44,  8){\line(1,0){33}}
\put( 81, 5){\circle{10}}
\put( 86, 5){\line(1,0){29}}
\put(120, 5){\circle{10}}
\put(125, 5){\line(1,0){20}}
\put(155, 2){$\cdot$}
\put(165, 2){$\cdot$}
\put(175, 2){$\cdot$}
\put(185, 5){\line(1,0){20}}
\put(210, 5){\circle{10}}
\put( 38, 15){$\sg$}
\put( 77, 15){$s_1$}
\put(117, 15){$s_2$}
\put(204, 15){$s_{n{-}1}$}
\end{picture}
\end{center}
Furthermore, $s_1,\ldots,s_{n-1}$ together with $u:=\sg s_1 \sg$ are Coxeter
generators for $W$ where $u,s_1$ commute with each other. As in
\cite[\S 5.5]{GePf}, we have a labelling of the irreducible characters
of $\tilde{W}$ by pairs of partitions $(\al,\bt)$ such that $|\al|+|\bt|=n$.
We write this as
\[ \Irr(\tilde{W})=\{\thi^{\alpha,\beta} \mid (\al,\bt) \vdash n\};\]
for example, the pair $((n),-)$ labels the trivial character and
$(-,(1^n))$ labels the sign character. Let $(\al,\bt)\vdash n$ and denote
by $\phi^{\alpha,\beta}$ the restriction of $\thi^{\alpha,\beta}$ from
$\tilde{W}$ to $W$. It is well-known that $\phi^{\alpha,\beta}=
\phi^{\beta,\alpha}\in \Irr(W)$ if $\alpha \neq \beta$ and
$\phi^{\alpha,\alpha}=\phi^{\alpha,+}+ \phi^{\alpha,-}$ where
$\phi^{\alpha,\pm}\in \Irr(W)$. Furthermore, all
irreducible characters of $W$ arise in this way.

\begin{rem}   \label{bndn1}
In the above setting, let $C \subseteq W$ be a
conjugacy class of involutions which is invariant under $\sg$. Then there
exists a parabolic subgroup $\tilde{W}_I \subseteq \tilde{W}$, where
$I \subseteq \{\sg,s_1, \ldots,s_{n-1}\}$, such that the following conditions
are satisfied:
\begin{itemize}
\item[($*$)] The longest element $w_I\in \tilde{W}_I$ is central in
$\tilde{W}_I$ and we have $w_I \in C$.
\end{itemize}
This is seen as follows.
By \cite[3.2.10]{GePf}, $C$ contains an element of the form
$w_{I'}$ where $I' \subseteq \{u,s_1,\ldots,s_{n-1}\}$ and $w_{I'}$ is
the longest element in the parabolic subgroup $W_{I'}\subseteq W$;
furthermore, $w_{I'}$ is central in $W_{I'}$.  If both $u$ and $s_1$
belong to $I'$, then we certainly have $w_{I'}=w_I$ where $I\subseteq
\{\sg,s_1,\ldots,s_{n-1}\}$ is the subset obtained by replacing $u$ by
$\sg$ in $I'$. Then ($*$) holds. Otherwise, since $C$ is invariant under
$\sg$, we can assume without loss of generality that $I' \subseteq
\{s_1,s_2,\ldots,s_{n-1}\}$. Then  we can set $I=I'$. Again, ($*$) holds.

In the first case, it is automatically true that $\sg(w_I)=w_I$.
In the second case, using \cite[3.4.12]{GePf}, we can even assume that
$I' \subseteq \{s_2,s_3,\ldots,s_{n-1}\}$ and then set $I=I'$. Hence,
in both cases, $I$ can actually be chosen such that $\sg(w_I)=w_I$.
\end{rem}

\begin{rem}   \label{bndn2}
Let $C\subseteq W$ and $I \subseteq \{\sg,s_1,
\ldots,s_{n-1}\}$ be as in Remark~\ref{bndn1}. We obtain the corresponding
$W$-module $V_C$ (as constructed in Section~\ref{sec:invmod}), with
character given by
\[\rho_C=\Ind_{C_W(w_I)}^W(\eps_I)\]
where $\eps_I \colon C_W(w_I) \rightarrow \{\pm 1\}$ is a certain
linear character of $C_W(w_I)$; see Lemma~\ref{luvo1}. Let
$\tilde{\rho}_C$ be the extension of $\rho_C$ to $\tilde{W}$.  Now, as in
Section~\ref{sec:invmod}, we have $C_{\tilde{W}}(w_I)= N_{\tilde{W}}
(\tilde{W}_I)=\tilde{W}_I \rtimes \tilde{Y}$ where $\tilde{Y}$ is a group
consisting of certain distinguished left coset representatives of
$\tilde{W}_I$ in $\tilde{W}$. Let $\tilde{\eps}_I\colon C_{\tilde{W}}(w_I)
\rightarrow \{\pm 1\}$ be the linear character such that
\[ \tilde{\eps}_I(yw')=(-1)^{l(w')-l_\sg(w')} \qquad \mbox{for
   all $y \in \tilde{Y}$ and $w'\in \tilde{W}_I$},\]
where $l_\sg(w)$ denotes the number of occurrences of the generator $\sg$
in a reduced expression for $w' \in \tilde{W}$. Note that $\tilde{\eps}_I$
is an extension of the linear character $\eps_I$ of $C_W(w_I)$ which is
used in the construction of $\rho_C$. Then we have
\[ \tilde{\rho}_C=\Ind_{C_{\tilde{W}}(w_I)}^{\tilde{W}}(\tilde{\eps}_I).\]
(This is seen by an argument entirely analogous to that in
Lemma~\ref{luvo1}.)
\end{rem}

\begin{lem}   \label{main2d}
 In the above setting, we have
 $\langle\thi^{\alpha,\beta},\tilde{\rho}_C\rangle=0$ unless
 $|\alpha|>|\beta|$.
\end{lem}

\begin{proof}
First recall that $\tilde{W}$ has a normal subgroup $\tilde{N}=\langle \sg_1,
\sg_2,\ldots, \sg_n \rangle$ where $\sg_1:=\sg$ and $\sg_i:=s_{i-1}\sg_{i-1}
s_{i-1}$ for $2\leq i \leq n$. Let $0\leq k\leq n$ and $\eta \colon \tilde{N}
\rightarrow \{\pm 1\}$ be a linear character which takes value $1$ on
exactly $k$ elements of $\{\sg_1,\sg_2,\ldots,\sg_n\}$ and value $-1$,
otherwise.  Then the usual construction of $\Irr(\tilde{W})$ via Clifford
theory (see \cite[5.5.4]{GePf}) shows that
\[ \big\langle \thi^{\alpha,\beta},\Ind_{\tilde{N}}^{\tilde{W}}(\eta)
   \big\rangle\neq 0 \qquad \Longleftrightarrow \qquad
   \alpha \vdash k \mbox{ and } \beta \vdash n-k.\]
Assume now that $(\alpha,\beta)\vdash n$ is such that $\langle \thi^{\alpha,
\beta},\tilde{\rho}_C\rangle \neq 0$.  Let $k=|\alpha|$ and $\eta$ be a
linear character of $\tilde{N}$ as above. Then we also have
\[ \big\langle \Ind_{\tilde{N}}^{\tilde{W}}(\eta),\tilde{\rho}_C\big\rangle
   \neq 0.\]
Now, by Remark~\ref{bndn2}, $\tilde{\rho}_C$ is obtained by inducing a
certain linear character from $C_{\tilde{W}}(w_I)$ to $\tilde{W}$. Hence,
using the Mackey formula, we conclude that there exists some $x \in
\tilde{W}$ such that $\eta^x$ and $\tilde{\eps}_I$ have the same
restriction to $\tilde{N} \cap C_{\tilde{W}}(w_I)$. Note that $\eta^x$ is
a linear character of $\tilde{N}$ similar to $\eta$, that is, it takes
value $1$ on exactly $k$ elements of $\{\sg_1,\sg_2,\ldots,\sg_n\}$ and
value $-1$, otherwise. Now let $\tilde{W}_I \subseteq \tilde{W}$ be a
parabolic subgroup such that Remark~\ref{bndn1}($*$) holds. Then
$\tilde{W}_I$ is of type $B_m \times A_1 \times \cdots \times A_1$ where
$m \geq 0$ and where the number of factors of type $A_1$ is strictly smaller
than $n/2$. Consequently, the element $w_I$ has the following form:
\[ w_I=\sg_1\sg_2\cdots \sg_m s_{m+1}s_{m+3}
   \cdots s_{m+2l-1} \quad \mbox{where} \quad 0 \leq l <n/2.\]
It follows that $\tilde{N}\cap C_{\tilde{W}}(w_I)$ is generated by the
elements
\begin{align*}
&\sg_1,\;\;\sg_2,\;\;\ldots,\;\;\sg_m,\\
&\sg_{m+1}\sg_{m+2},\;\;\sg_{m+3} \sg_{m+4},\;\; \ldots,\;\;
\sg_{m+2l-1}\sg_{m+2l},\\
&\sg_{m+2l+1},\;\;\sg_{m+2l+2},\;\;\ldots,\;\; \sg_{n}.
\end{align*}
(To see this, it may be useful to work with the usual realisation
of $\tilde{W}$ as a group of monomial matrices. In this realisation,
each $\sigma_i$ is represented by a diagonal matrix where the $i$th
diagonal entry is $-1$ and all other diagonal entries are $1$.)
By the construction of $\tilde{\eps}_I$, we have
\begin{alignat*}{2}
\tilde{\eps}_I(\sg_i)&=&\,1 & \qquad (1 \leq i \leq m),\\
\tilde{\eps}_I(\sg_{m+2i-1}\sg_{m+2i})&=&\,-1 &\qquad (1\leq i \leq l),\\
\tilde{\eps}_I(\sg_i)&=&\,1 &\qquad (i \geq m+2l+1).
\end{alignat*}
(The first and third equalities are clear. As far as the second equality
is concerned, note that $\sg_{m+2i-1}\sg_{m+2i}=s_{m+2i-1}y$ where
$s_{m+2i-1} \in I$ and $y= \sg_{m+2i-1} s_{m+2i-1}\sg_{m+2i-1} \in
\tilde{Y}$; so, by Remark~\ref{bndn2}, we have $\tilde{\eps}_I
(s_{m+2i-1})=-1$ and $\tilde{\eps}_I(y)=1$.) Since $\eta^x$ and
$\tilde{\eps}_I$ have the same restriction to $\tilde{N} \cap
C_{\tilde{W}}(w_I)$, we must have $\eta^x(\sg_i)=1$ for $1 \leq i\leq m$
and $\eta^x(\sg_i)=1$ for $i\geq m+2l+1$. Furthermore, for each $i
\in\{1,\ldots,l\}$, we must have either $\eta^x(\sg_{m+2i-1})=1$ or
$\eta^x(\sg_{m+2i})=1$. Hence, $\eta^x$ takes value $-1$ on exactly
$l$ elements in $\{\sg_1,\sg_2,\ldots,\sg_n\}$. Thus, we conclude that
$k=n-l>n/2$, as required.
\end{proof}

Recall Lusztig's notion of special characters for $W$ and of the finite
$2$-group associated to the family of a special character; see
\cite[Chap.~4]{LuB}.

\begin{prop}   \label{prop:signD2}
 Let $W=D_n$ ($n \geq 2$) and $\sg$ the graph automorphism of order~$2$,
 as above. As a class function on $W.\sg$, we have
 $$\tilde\rho=\sum_{(\al,\bt)}2^{d(\al,\bt)}\,\thi^{\al,\bt},$$
 where the sum is over all pairs $(\al,\bt)\vdash n$ such that $|\al|>|\bt|$
 and $\phi^{\al,\bt}$ is special, and $2^{d(\al,\bt)}$ is the order of the
 finite group attached to the family of $\phi^{\al,\bt}$. All the extended
 characters $\thi^{\alpha,\beta}$ appearing in the above sum are
 the preferred extensions as in Remark~\ref{rempref}.
\end{prop}

\begin{proof}
First note that, in the decomposition of $\tilde{\rho}$
as a class function on $W.\sg$, there will only be characters $\thi^{\alpha,
\beta}$ where $\alpha \neq \beta$. Now, by \cite[3.3]{Ko}, $\phi^{\al,\bt}$
($\alpha \neq \beta$) is a constituent of $\rho$ if and only if it is a
special character, and in that case, the multiplicity is given by
$2^{d(\al,\bt)}$. In combination with Lemma~\ref{main2d} (and
Remark~\ref{firstrem}), this immediately yields the above formula for
$\tilde{\rho}$.

It remains to check the statement concerning the preferred extensions.
Let $(\alpha,\beta)\vdash n$ be such that $\alpha \neq \beta$ and consider
the corresponding character $\phi^{\alpha, \beta} \in \Irr(W)$. As in
\cite[4.6]{LuB}, we have an associated symbol with two rows of equal
length
\[ \left(\begin{array}{c} \lambda_1<\lambda_2 <\ldots <\lambda_m\\
\mu_1<\mu_2 <\ldots <\mu_m\end{array}\right)\]
where the upper row is associated with $\alpha$ and the lower row is
associated with $\beta$. Recall from Remark~\ref{rempref} that 
$\thi^{\alpha,\beta}\in\Irr(\tilde{W})$ is the preferred extension 
of $\phi^{\alpha,\beta}$ if the above symbol has the following property: 
the smallest entry which appears in only one row appears in the lower row. 
(Note that $\phi^{\alpha,\beta}=\phi^{\beta,\alpha}$, so the two rows of
the symbol can be interchanged as far as $\Irr(W)$ is concerned; however, 
when we consider extensions to $\tilde{W}$, then $\thi^{\alpha,\beta}\neq 
\thi^{\beta,\alpha}$ and so the order of the two rows does matter.) Now, 
by \cite[4.6]{LuB}, the condition for $\phi^{\alpha,\beta}$ to be special 
is that either
\[ \lambda_1\leq \mu_1 \leq \lambda_2 \leq \mu_2 \leq \ldots \leq
\lambda_m \leq \mu_m\]
or
\[ \mu_1\leq \lambda_1 \leq \mu_2 \leq \lambda_2 \leq \ldots \leq
\mu_m \leq \lambda_m.\]
Hence, if $\phi^{\alpha,\beta}$ is special, then $\thi^{\alpha,\beta}$ is
the preferred  extension precisely when we are in the second case.
But in this case, we have $\sum_{1\leq i \leq m} \mu_i < \sum_{1\leq i
\leq m} \lambda_i$ which immediately implies that $|\beta|< |\alpha|$.
Thus, all the characters $\thi^{\alpha,\beta}$ appearing in the
decomposition of $\tilde{\rho}$ are preferred extensions.
\end{proof}

\begin{rem}
Assume that $n$ is odd. Then one can give a different proof based
on Proposition~\ref{prop:inn}, as follows. We need to specify certain
involutions in $W$. For this, recall that $W$ can be realized as a subgroup
of index two in the group of signed permutations of $\{1,\ldots,n\}$. For
non-negative integers $j,k,l$ with $n=2j+k+l$ let $t_{j,k,l}$ denote the
involution in $W$ which acts by the permutation $(1,2)\cdots(2j-1,2j)$ on
the first $2j$ letters, and by $-1$ on the last $l$, and write $C(j,k,l)$
for its class. Then by \cite[3.3]{Ko}, $\phi^{\al,\bt}$ occurs in
$\rho_{C(j,k,l)}$ if and only if $\al\vdash j+k+l$, $\bt\vdash j$. Now
the character $\thi^{\al,\bt}$ has sign $(-1)^{|\bt|}$ on the central
element $w_0\sg$, while the length of $t_{j,k,l}$ clearly satisfies
$l(t_{j,k,l}) \equiv j\pmod2$, so we may conclude by
Proposition~\ref{prop:inn}.
\end{rem}

\section{Exceptional types} \label{sec:exc}

We now consider the various cases of exceptional type.

\subsection{$I_2(m)$, $m$ odd, with $o(\sg)=2$}
Let $(W,S)$ be the Coxeter group of type $I_2(m)$, $m\ge3$ odd, and $\sg$ the
graph automorphism of $W$ interchanging the two generating reflections in $S$.
Then $\Irr(W)$ consists of the trivial character $\phi_{1,0}$, the
sign character $\phi_{1,m}$ and $(m-1)/2$ irreducible characters $\phi_{2,k}$,
$1\le k\le (m-1)/2$, of degree~2.

\begin{prop}   \label{prop:I2modd}
 Let $W$ be of type $I_2(m)$, $m$ odd, and $o(\sg)=2$. As a class
 function on $W.\sg$, we have
 $$\tilde\rho= \thi_{1,0}-\thi_{1,m}-\sum_{1\leq k \leq (m-1)/2}
 \thi_{2,k},$$
 where $\thi_*$ denotes the extension of $\phi_*$ to $W.\sg$ with positive
 value on $w_0\sg$.
\end{prop}

\begin{proof}
By \cite[Prop.~3.5]{Mar} we have
$$\rho=\sum_{\phi\in\Irr(W)}\,\phi.$$
According to \cite[Tab.~1]{Mar} all constituents except for $\phi_{1,0}$ occur
in $\rho_C$, where $C$ is the unique conjugacy class of involutions of $W$.
Thus, by Proposition~\ref{prop:inn}, $\langle\tilde\rho_C,\thi\rangle=
\langle\rho_C,\phi\rangle=1$ for all of these, with $\thi$ the extension
with $(-1)^{l(t)}\thi(w_0\sg)>0$, where $t\in C$. Since the generating
reflections lie in $C$ and have length~1, the claim follows.
\end{proof}

\subsection{$I_2(m)$, $m$ even, with $o(\sg)=2$}
Let $(W,S)$ be the Coxeter group of type $I_2(m)$, $m\ge4$ even, and $\sg$ the
exceptional graph automorphism of $W$ interchanging the two generating
reflections in $S$. Then $W$ has the trivial character $\phi_{1,0}$, the
sign character $\phi_{1,m}$, two further linear characters $\phi_{1,m/2}'$
and $\phi_{1,m/2}''$, and $m/2-1$ irreducible characters $\phi_{2,k}$. All of
these except for the two linear characters $\phi_{1,m/2}'$ and
$\phi_{1,m/2}''$ are $\sg$-invariant.

\begin{prop}   \label{prop:I2m}
 Let $W$ be of type $I_2(m)$, $m$ even, and $o(\sg)=2$. As a class function
 on $W.\sg$, we have
 $$\tilde\rho= \thi_{1,0}+\thi_{1,m},$$
 where $\thi_{1,k}$ denotes the extension of $\phi_{1,k}$ to
 $W.\sg$ with positive value on $\sg$, for $k=0,m$.
\end{prop}

\begin{proof}
First consider the case where $m\equiv2\pmod4$. Then by \cite[Prop.~3.5]{Mar},
$$\rho=\phi_{1,0}+\phi_{1,m/2}'+\phi_{1,m/2}''+\phi_{1,m}
       +\sum_{1\leq k \leq (m-2)/4}2\phi_{2,2k-1}.$$
As pointed out above, the second and third linear character are interchanged
by $\sg$. Moreover, according to \cite[Tab.~1]{Mar} for any $k$ the two copies
of $\phi_{2,2k-1}$ lie in submodules corresponding to the involution classes
of the two (non-conjugate) elements in $S$, which are interchanged
by $\sg$. Thus the extension of their sum vanishes on $W.\sg$.
Similarly, when $m\equiv0\pmod4$ then by \cite[Prop.~3.5]{Mar} we have
$$\rho=\phi_{1,0}+\phi_{1,m}+\sum_{1 \leq k \leq m/4}2\phi_{2,2k-1}.$$
The same argument as before applies.
\end{proof}

\begin{exmp}
The above covers in particular the exceptional graph automorphisms of the
two Weyl groups of type $B_2$ and $G_2$.
\begin{enumerate}
 \item Let $W$ be of type $B_2$ with $\sg$ interchanging the two generating
  reflections in $S$. Here $\Irr(W)$ has five elements, among which are the
  trivial representation $\phi_{1,0}$, the sign representation $\phi_{1,4}$
  and the reflection representation $\phi_{2,1}$. We denote by $\thi_{1,k}$
  the extension of $\phi_{1,k}$ to $\tilde W=\langle W,\sg\rangle$ taking
  value~1 on $\sg$, for $k=0,4$. Then by Proposition~\ref{prop:I2m} we have
  $$\tilde\rho=\thi_{1,0}+\thi_{1,4}.$$
 \item Let $W$ be of type $G_2$ with $\sg$ interchanging the two generating
  reflections in $S$. Here $\Irr(W)$ has six elements, among which are the
  trivial representation $\phi_{1,0}$, the sign representation $\phi_{1,6}$,
  and the reflection representation $\phi_{2,1}$. We denote by $\thi_{1,k}$
  the extension of $\phi_{1,k}$ to $\tilde W=\langle W,\sg\rangle$ taking
  value~1 on $\sg$, for $k=0,6$. Then by Proposition~\ref{prop:I2m} we have
  $$\tilde\rho=\thi_{1,0}+\thi_{1,6}.$$
 \end{enumerate}
\end{exmp}

\subsection{$D_4$ with $o(\sg)=3$}
Let $(W,S)$ be the Coxeter group of type $D_4$, and $\sg$ the exceptional
graph automorphism of $W$ of order~3. Here $\Irr(W)$ has 13 elements, seven
of which extend to $\tilde W=\langle W,\sg\rangle$. We denote the
irreducible characters of $W$ by pairs of partitions, as in the previous
section.

\begin{prop}   \label{prop:D4}
 Let $W$ be of type $D_4$ and $o(\sg)=3$. As a class function on $W.\sg$, we
 have
 $$\tilde\rho=\thi^{-,4}+\thi^{-,1^4}+\thi^{1,3}+\thi^{1,1^3}+2\thi^{1,21},$$
 where $\thi^*$ denotes the extension of $\phi^*$ to $W.\sg$ which
 takes positive integral value on $\sg$. (These are the preferred
 extensions as in Remark~\ref{rempref}.)
\end{prop}

\begin{proof}
By \cite[Thm.~1.4]{Mar} the decomposition of $\rho$ is given by
$$\rho=\phi^{-,4}+\phi^{-,1^4}+\phi^{1,3}+\phi^{1,1^3}+2\phi^{1,21}+
\phi^{2,+}+\phi^{2,-}+\phi^{1^2,+}+\phi^{1^2,-}+\phi^{-,31}+\phi^{-,21^2}.$$
Here, the first five constituents extend to $\tilde W$, while the last six are
permuted in triples by $\sg$. Now the values of the extensions $\thi^*$ of the
first five constituents on $\sg$ add up to $\tilde\rho(\sg)=8=|\bI^\sg|$, thus
these are exactly the constituents of $\tilde\rho$ on $W.\sigma$.
\end{proof}

\subsection{$F_4$ with $o(\sg)=2$}
Let $(W,S)$ be the Coxeter group of type $F_4$, and $\sg$ the exceptional
graph automorphism of $W$ of order~2. Here $\Irr(W)$ has 25 elements, eleven
of which extend to $\tilde W=\langle W,\sg\rangle$.

\begin{prop}   \label{prop:2F4}
 Let $W$ be of type $F_4$ and $o(\sg)=2$. As a class function on $W.\sg$, we
 have
 $$\tilde\rho=\thi_{1,0}+\thi_{9,2}+\thi_{9,10}+\thi_{1,24}
              +\thi_{12,4}-2\thi_{6,6}''+\thi_{6,6}',$$
 where $\thi_*$ are the extensions of $\phi_*$ printed in \cite[Tab.~1]{GM}.
\end{prop}

\begin{proof}
By the result of Casselman \cite[p.~39]{Cas}, the decomposition of $\rho$ is
given by
$$\begin{aligned}
\rho=&\,\phi_{1,0}+\phi_{9,2}+\phi_{8,3}'+\phi_{8,3}''+\phi_{8,9}'+
\phi_{8,9}''+ \phi_{9,10}+\phi_{1,24}+2\phi_{4,1}+2\phi_{4,13}\\
  &+3\phi_{12,4}+\phi_{9,6}'+2\phi_{6,6}''+\phi_{9,6}''+\phi_{6,6}'.
\end{aligned}$$
Of these, the characters of degree~8 and the characters $\phi_{9,6}',
\phi_{9,6}''$ do not extend to $\tilde W$. By \cite[p.39]{Cas} the two copies
of representations $\phi_{4,1}$ and $\phi_{4,13}$ lie in two submodules
corresponding to non-conjugate involutions interchanged by $\sg$,
so $W.\sg$ has trace~0 on their sum. By explicit computation of traces on the
submodules of dimensions~18 and~72 corresponding to the two $\sg$-fixed
non-central classes of involutions one finds the stated decomposition.
\end{proof}

\subsection{$E_6$ with $o(\sg)=2$}
Let $(W,S)$ be the Coxeter group of type $E_6$, and $\sg$ the graph
automorphism of $W$ of order~2; it is given by conjugation with the
longest element $w_0$. Here $\Irr(W)$ has 25 elements, all of which
extend to $\tilde W=\langle W,\sg\rangle$.

\begin{prop}   \label{prop:E6}
 Let $W$ be of type $E_6$ and $o(\sg)=2$. As a class function on $W.\sg$,
 we have
 $$\begin{aligned}
   \tilde\rho= &\,\thi_{1,0}+\thi_{6,1}+\thi_{20,2}+\thi_{64,4}+\thi_{60,5}
      +\thi_{81,6}+\thi_{24,6}+\thi_{81,10}+\thi_{60,11}+\thi_{24,12}\\
     &+\thi_{64,13}+\thi_{20,20}+\thi_{6,25}+\thi_{1,36}+ 2\thi_{30,3}
      +2\thi_{30,15}+2\thi_{80,7}+\thi_{90,8}+\thi_{10,9},
 \end{aligned}$$
 where $\thi_*$ denotes Lusztig's preferred extension of $\phi_*$ to 
 $W.\sg$ (see Remark~\ref{rempref}).
\end{prop}

\begin{proof}
By the result of Casselman \cite[p.~40]{Cas}, the decomposition of $\rho$ is
given by
$$\begin{aligned}
\rho=&\,\phi_{1,0}+\phi_{6,1}+\phi_{20,2}+\phi_{64,4}+\phi_{60,5}+\phi_{81,6}
      +\phi_{24,6}+\phi_{81,10}+\phi_{60,11}+\phi_{24,12}\\
     &+\phi_{64,13}+\phi_{20,20}+\phi_{6,25}+\phi_{1,36}+ 2\phi_{30,3}
      +2\phi_{30,15}+2\phi_{80,7}+\phi_{90,8}+\phi_{10,9}.
\end{aligned}$$
So the claim follows from Proposition~\ref{prop:inn} and inspection
of the {\sf Chevie} \cite{Chv} tables of $W$ and $\tilde{W}$.
\end{proof}

\begin{rem} Assume that $W$ is an irreducible Weyl group and $\sg$ is 
ordinary (see Remark~\ref{rempref}). Then the results in this
and the previous section show that, in the decomposition of $\tilde{\rho}$
as a class function on $W.\sg$, all multiplicities are $\geq 0$ if we
choose Lusztig's preferred extensions of the characters in $\Irr(W)^\sg$
as in Remark~\ref{rempref}.
\end{rem}

\section{Frobenius--Schur indicators and Fourier matrices} \label{sec:FS}

We shall now interpret the multiplicity formulae for $\tilde{\rho}$
obtained in the previous two sections in terms of Fourier matrices.

\subsection{}
The Fourier matrix associated with $W, \sg$ is a matrix with rows and
columns labelled by two finite, purely combinatorially defined sets
$\bar{X}(W,\sg)$ and $\Lambda(W,\sg)$, respectively. We have $|\bar{X}(W,
\sg)|=|\Lambda(W,\sg)|$ and there is a well-defined injection $\Irr(W)^\sg
\hookrightarrow \Lambda(W,\sg)$. Let us briefly recall how these are
defined.

If $W$ is a Weyl group and $\sg$ is ordinary  (see Remark~\ref{rempref}),
then $\bar{X}(W,\sg)$ is the set defined in \cite[4.21.11]{LuB}, which is 
in bijection with the set of unipotent characters of a corresponding finite 
group of Lie type by \cite[Main Theorem~4.23]{LuB}. The set $\Lambda(W,
\sg)$ is obtained by choosing a complete set of representatives for the 
$M$-orbits on the set $X(W,\sg)$ defined in \cite[4.21.12]{LuB} (with $M$ 
as in \cite[4.16]{LuB}). Then \cite[4.21.14]{LuB} gives rise to the natural 
injection $\Irr(W)^\sg \hookrightarrow \Lambda(W,\sg)$. The corresponding 
Fourier matrix is obtained as follows. The entry for $\bar{x} \in \bar{X}
(W,\sg)$ and $y \in \Lambda(W,\sg)$ is given by evaluating the canonical 
pairing $\bar{X} (W,\sg) \times X(W,\sg)\rightarrow \bar{\QQ}_l$ in 
\cite[4.21.13]{LuB} on $(\bar{x}, y)$ and then multiplying the result by 
a sign $\Delta(\bar{x})=\pm 1$, as defined in \cite[p.~124/126]{LuB}. If $W$ 
is not a Weyl group, or if $W$ is of type $B_2$, $G_2$ or $F_4$ and $\sg$ 
is not ordinary, the Fourier matrix has been described by heuristic methods 
in \cite{exot} and \cite{GM}, respectively.

Now, if $\sg=\id$, there is a canonical identification $\Lambda(W,\sg)=
\bar{X}(W,\id)$, hence the Fourier matrix is canonically defined in this 
case. Otherwise, there are certain choices involved in the definition of 
the Fourier matrix. As far as the entries corresponding to the image of
$\Irr(W)^\sg$ in $\Lambda(W,\sg)$ are concerned, these depend precisely
on the choices of extensions of the characters in $\Irr(W)^\sg$ to the
coset $W.\sg$.

The following result provides the promised interpretation of the
decomposition of $\tilde{\rho}$ in the case where $W$ and $\sg$ arise from a 
simple algebraic group $G$ and an endomorphism $F\colon G \rightarrow G$ 
as in Section~\ref{sec:intro}. In this case, the Fourier matrix 
describes the multiplicities
\[ \langle \chi,R_{\thi}\rangle \qquad \mbox{where} \qquad
R_{\thi}:=\frac{1}{|W|} \sum_{w \in W} \thi(w\sg) R_{T_w,1},\]
for any $\chi \in \Uch(G^F)$ and $\phi \in \Irr(W)^\sg$; see 
\cite[Main Theorem 4.23]{LuB}.

\begin{thm}   \label{thm:main}
 Let $W$ be a finite irreducible Weyl group with generating set $S$, and
 $\sg\colon W \rightarrow W$ a non-trivial automorphism with $\sg(S)=S$. For
 each $\phi \in \Irr(W)^\sg$, we fix an extension $\thi$ to $W.\sg$. Then
 the $\{1,-1, 0\}$-vector of Frobenius--Schur indicators of the unipotent
 characters of the corresponding twisted groups of Lie type (indexed by the
 set $\bar{X}(W,\sg)$) is mapped under Fourier transform and restriction to
 the image of $\Irr(W)^\sg$ in $\Lambda(W,\sg)$ onto the vector of
 multiplicities in the decomposition of $\tilde{\rho}$ as a class function
 on $W.\sg$.
\end{thm}

\begin{proof}
First note that the assertion does not depend on the actual choices
of the extensions $\thi$, as long as we use the same extensions both for
the Fourier matrix and for the decomposition of $\tilde{\rho}$. Now we
consider the various cases. 

Let first $W=\fS_n$ with the non-trivial graph automorphism $\sg$. We have
a natural parametrisation $\Uch(\GU_n(q))=\{\chi_\al \mid \al \vdash n\}$
where the degree of $\chi_\al$ is given in terms of a well-defined polynomial 
in $q$; see Lusztig \cite[\S 9]{Lu1} for further details. Let $A_\al$ 
denote the degree in $q$ of the degree polynomial of $\chi_\al$; its 
order of vanishing at $q=0$ is the invariant $a_\al$ already mentioned in
Remark~\ref{rempref}. We claim that 
\begin{itemize}
\item[($\dagger$)] $\chi_\al$ has Frobenius--Schur indicator $(-1)^{a_\al+
A_\al}$. 
\end{itemize}
This is seen as follows. By Ohmori \cite{Oh} (see also \cite{Lu02}), the
Frobenius--Schur indicator of $\chi_\al$ is given by $(-1)^{\lfloor k/2
\rfloor}$ where $\kappa$ is the $2$-core of $\al$. Now $\chi_\al$ lies in the 
Harish-Chandra series of the (unique) cuspidal unipotent character $\la$ 
of $\GU_k(q)$ (see \cite[9.6]{Lu1}). In particular (\cite[7.8]{Lu1}), its 
degree polynomial is divisible by the same power of $q-1$ as that of $\la$,
viz.\ $q-1$ to the power ${\lfloor k/2\rfloor}$. Now note that unipotent 
characters of general linear groups all lie in the principal series, so
their degree polynomials are not divisible by $q-1$; thus, by Ennola 
duality (see \cite[9.5]{Lu1}) the degree polynomials of unipotent characters 
of unitary groups are not divisible by $q+1$. So the only odd degree factors 
of these degree polynomials are $q$ and $q-1$. In particular, we obtain 
$A_\al\equiv a_\al+ {\lfloor k/2\rfloor}\pmod 2$. Thus ($\dagger$) is
proved. Now consider the corresponding Fourier matrix. Since all families 
are singletons, this matrix is diagonal with $\pm 1$ on the diagonal. By
the description of the pairing $\bar{X}(W,\sg) \times \Lambda(W,\sg)$ in
\cite[4.19]{LuB} and that of the $\Delta$-function in \cite[p.~124]{LuB}
(see also p.~235 in the proof of \cite[Prop.~7.6]{LuB}), we have $\langle 
\chi_\al,R_{\thi_\al}\rangle=(-1)^{A_\al}$ where $\thi_\al$ is the 
extension of $\phi_\al$ in which $\sg$ acts as the longest element $w_0$. 
Passing to the preferred extension, we conclude 
that the corresponding diagonal entry of the Fourier matrix is 
$(-1)^{a_\al+A_\al}$. Hence, multiplication with the Fourier matrix indeed 
gives the multiplicities computed in Proposition~\ref{prop:signA}. \par

In Type $D_n$, $o(\sg)=2$, all Frobenius--Schur indicators of the unipotent
characters of $\tw2D_n(q)$ are equal to~$+1$ by Lusztig \cite[1.13]{Lu02}.
The Fourier matrix is described in \cite[4.18]{LuB}. It is a block diagonal
matrix with blocks corresponding to the various $\sg$-stable families of
$\Irr(W)$. It is clear that the all $1$ vector transforms to the vector
with value $\pm 2^d$ at the image of a $\sg$-stable special character in
$\Lambda(W, \sg)$ (where $2^d$ is the order of the finite group associated
with the family containing the given special character) and $0$ otherwise.
We need to show that, if we choose preferred extensions as in
Proposition~\ref{prop:signD2}, then the above values are always $+2^d$.
For this purpose, it will be enough to show that the entries in the Fourier
matrix corresponding to the preferred extension of a $\sg$-stable special
character are all $\geq 0$. But this follows from the description in
\cite[4.18]{LuB}. Indeed, let $\phi\in\Irr(W)$ be a $\sg$-stable special
character.
As in \cite[4.6]{LuB}, we have a corresponding nondegenerate symbol $Z$
with two rows of equal length. Let $Z_1$ be the set of singles in $Z$ and
$Z_2$ be the set of doubles in $Z$. Furthermore, let $Z_1=M_0 \amalg M_0'$
be the partition defined by the two rows of $Z$, where $M_0,M_0'$ are
distinguished one from another by the inequality $\sum_{x \in M_0} x<
\sum_{x \in M_0'} x$, as in \cite[p.~93]{LuB}. Then, as in
\cite[p.~117]{LuB}, the two extensions of $\phi$ are labelled by the symbols
\[\left(\begin{array}{c} Z_2 \amalg M_0'\\
    Z_2 \amalg M_0 \end{array}\right) \qquad \mbox{and} \qquad
  \left(\begin{array}{c} Z_2 \amalg M_0\\
    Z_2 \amalg M_0' \end{array}\right),\]
respectively. Now recall that a preferred extension is characterised by
the condition that the smallest element of $Z_1$ has to appear in the
lower row of the symbol associated with the extension. Hence, using
an argument as in the proof of Proposition~\ref{prop:signD2}, it
immediately follows that the preferred extension of a special character
is labelled by the first of the above two symbols (the one where $M_0$
is in the lower row). The formula at the end of \cite[4.18]{LuB} then
shows that the corresponding entries in the Fourier matrix are all
$\geq 0$, as required. Thus, we obtain the vector of multiplicities in
$\tilde{\rho}$ as computed in Proposition~\ref{prop:signD2}.
\par
For $W=E_6$ with $\sg$ the non-trivial graph automorphism, it is shown
in \cite[5.6, 6.5]{Ge03} that all unipotent characters of $\tw2E_6(q)$ have
Frobenius--Schur indicator $+1$, except for the two cuspidal unipotent
characters denoted $\tw2E_6[\theta]$, $\tw2E_6[\theta^2]$ where the 
indicator is $0$, and for the two characters lying above the cuspidal 
unipotent character of $\GU_6(q)$, with indicator $-1$. The Fourier 
matrices are described in \cite[Thm.~1.15]{LuE} (see also 
\cite[Prop.~7.11]{LuB}), with respect to the extensions of the $\sg$-stable
characters of $W$ which have positive value on $w_0\sg$. Let $E_6(q)$ be 
the Chevalley group of split type $E_6$. Then there is a bijection $\chi 
\leftrightarrow \chi'$ between $\Uch(\tw2E_6(q))$ and $\Uch(E_6(q))$ such 
that 
\[ \langle \chi,R_{\thi}\rangle=\delta_\chi \langle \chi',R_\phi\rangle
\qquad \mbox{for all $\chi\in \Uch(\tw2E_6(q))$ and $\phi\in\Irr(W)^\sg$},\]
where $\delta_\chi=\pm 1$ is determined by the condition that the degree
polynomial of $\chi$ is obtained from that of $\chi'$ by formally replacing
$q$ by $-q$ and then multiplying by $\delta_\chi$. Hence, using the Fourier
matrix for $E_6(q)$ and the table of degree polynomials in 
\cite[p.~363]{LuB}, and then passing to the preferred extensions, the claim 
follows by comparison with Proposition~\ref{prop:E6}.
\par
For $W=D_4$, $o(\sg)=3$, the corresponding finite reductive group
$\tw3D_4(q)$ has five families of unipotent characters, with $1,1,4,1,1$
characters respectively (see \cite[1.17]{LuE}), and all Frobenius--Schur 
indicators are equal to~$1$ (see \cite[Tab.~1]{Ge03}). The Fourier 
matrix (see \cite[Thm.~1.18]{LuE} or \cite[Prop.~7.6]{LuB}) transforms this 
to the vector $(1,1,2,0,0,0,1,1)$, which upon comparing the labels, just gives 
the decomposition in Proposition~\ref{prop:D4}.
\par
Now consider the cases where $\sg$ is not ordinary.
Let first $W=B_2$, $o(\sg)=2$. There are
three families of unipotent characters for the corresponding finite reductive
group $\tw2B_2(q^2)$, with 1,2,1 characters respectively (see e.g.
\cite[2.3]{GM}), and Frobenius--Schur indicators $1,0,0,1$. The transform
under the Fourier matrix is thus $(1,0,0,1)$, where the entries equal to $1$
correspond to the class functions $\thi_{1,0},\thi_{1,4}$ on $W.\sg$,
respectively. By Proposition~\ref{prop:I2m} this is the vector of
multiplicities in $\tilde\rho$, as claimed. \par
For $W=G_2$, $o(\sg)=2$, the corresponding finite reductive group
$^2G_2(q^2)$ has three families of unipotent characters, with $1,6,1$
characters respectively, and Frobenius--Schur indicators $1,0,0,0,0,0,0,1$.
This is an eigenvector of the Fourier transform matrix given in
\cite[Thm.~5.4]{GM}, and the claim follows with Proposition~\ref{prop:I2m}.
\par
For $W=F_4$, $o(\sg)=2$, the corresponding finite reductive group
$\tw2F_4(q^2)$ has seven families of unipotent characters, with $1,1,1,1,
2,2,13$ characters respectively. The Frobenius--Schur indicators in the
1-element families are equal to~1, in the 2-element families equal to~0,
and on the 13-element family are given by $(1,1,1, 1,1,1,0,0,0,0,0,0,-1)$
(see \cite[\S7]{Ge03}). Multiplication with the Fourier transform matrix
\cite[Tab.~2]{GM} here gives the vector $(1,0,1,-2,0,0,0,0,-1,0,0,0,0)$,
which upon comparing labels with loc.\ cit.\ gives the claim by
Proposition~\ref{prop:2F4}.
\end{proof}

\begin{rem} 
Lusztig \cite[3.9]{Lu2} has shown that one can attach to each
$\chi \in \Uch(G^F)$ a corresponding ``eigenvalue of Frobenius'',
which is a root unity and will be denoted by $\Fr(\chi)$. The results on 
character fields for unipotent characters in \cite[\S 5]{Ge03} and 
\cite[\S 4]{GM} show that, for any $\chi \in \Uch(G^F)$, we have 
\[ \nu(\chi)=0 \qquad \Longleftrightarrow \qquad \Fr(\chi) \mbox{ is 
  non-real}.\]
Under the Fourier matrix, the unipotent characters transform to a new
set of class functions which are labelled by $\Lambda(W,\sg)$; see 
\cite[4.24.1]{LuB}. These class functions are called ``almost characters''
of $G^F$. One can also attach a Frobenius eigenvalue to any such almost
character (by using ``twisting operators'';  see \cite[\S 5]{GM} and the 
references there).

Now, it is not true in general that the vector of Frobenius--Schur 
indicators, multiplied by the Fourier-transform, has non-zero values only 
on the image of $\Irr(W)^\sg$ in $\Lambda(W,\sg)$. For example, in type 
$F_4$, we also find non-zero values for $F_4^I[1]$; in type $E_8$ for 
$E_8^I[1]$ (recall that, in the split case, we can identify $\bar{X}(W,\sg)=
\Lambda(W,\id)$); and in $\tw2F_4$ for the element of $\Lambda(W,\sg)$ 
denoted $\Psi_4$ in \cite[Tab.~2]{GM}. We would like to point out that 
in each case, this is one out of the two ``cuspidal'' elements in
$\Lambda(W,\sg)$ with attached Frobenius eigenvalue~$1$.

\par
We are not aware of any, even heuristical, explanation of this.
\end{rem}

\subsection{}
As in \cite[\S 6.4]{LV}, we shall now reformulate Theorem~\ref{thm:main}
without reference to Fourier matrices. One advantage of this reformulation
will be that it does not involve choices of extensions of $\sg$-stable 
characters of $W$. We will also be able to drop the assumption that $W$ 
is irreducible. Let $G$ be a connected reductive algebraic group over 
$\overline{\FF}_p$ and $F \colon G \rightarrow G$ be an endomorphism as in 
Section~\ref{sec:intro}. Recall the definitions of $\Uch(G^F)$ and 
$\CF_0(G^F)$. The map $F$ induces an automorphism $\sg \colon W 
\rightarrow W$ such that $\sg(S)=S$. Let $\tilde{W}=W \rtimes \langle \sg 
\rangle$. For any class function $f$ on the coset $W.\sg$, we define
\[ R_f:=\frac{1}{|W|} \sum_{w \in W} f(w\sg) \, R_{T_w,1}.\]
With this notation, we can now state:

\begin{cor}[Cf.\ Lusztig--Vogan \protect{\cite[6.4(b)]{LV}} for the case
$\sg=\id$] \label{cor1}
 Let $\tilde{\rho}$ be the character of the extended involution module,
 as in Section~\ref{sec:invmod}. Then
 \[ R_{\tilde{\rho}}=\bigl(\sum_{\chi \in \operatorname{Uch}(G^F)}
    \nu(\chi) \, \chi\bigr)_0\]
 where $\nu(\chi)$ denotes the Frobenius--Schur indicator of $\chi$.
\end{cor}

\begin{proof}
By \cite[Prop.~7.10]{DeLu}, the set $\Uch(G^F)$ is ``insensitive'' to the
centre of $G$, and similarly for $R_f$. Hence, we may assume without
loss of generality that $G$ is semisimple of adjoint type. We now proceed
by induction on $\dim G$, where we use a reduction argument analogous
to \cite[8.8]{LuB}.

If $G$ itself is simple, then the result follows from
Theorem~\ref{thm:main}, exactly as in \cite[6.4]{LV}. Next assume that we
have a non-trivial factorisation $G=G_1 \times G_2$ where both $G_1$
and $G_2$ are $F$-stable. Then $G^F=G_1^F \times G_2^F$ and the result
easily follows by induction. It remains to consider the following case: we
have $G=G_1 \times \cdots \times G_1$ (say, $d\geq 2$ factors) where $G_1$
is simple of adjoint type; furthermore, $F$ cyclicly permutes the factors
and $F^d(G_1)=G_1$. We then have a natural isomorphism $G^F \cong
G_1^{F^d}$ which preserves unipotent characters. Hence, the right hand
side of the desired equality is preserved under this isomorphism. It remains
to see what happens on the left hand side. By the definition of
$R_{\tilde{\rho}}$, this is easily reduced to a question about the
extended involution module, purely on the level of $W$ and $\sg$. Now,
we have $W=W_1\times \cdots \times W_1$ (where $W_1$ is the Weyl group of
$G_1$) and $\sg$ cyclicly permutes the factors such that $\sg^d(W_1)=W_1$.
We need to compare the characters $\tilde\rho$ of the extended involution module
for $W,\sg$ and $\tilde\rho_1$ of the one for $W_1,\sg^d$, respectively. \par
We may assume that for $(w_1,\ldots,w_d)\in W_1\times\cdots\times W_1=W$ we
have $\sg(w_1,\ldots,w_d)=(w_2,\ldots,w_d,\sg^d(w_1))$. The map
$$p:W\rightarrow W_1,\qquad (w_1,\ldots,w_d)\mapsto w_1\cdots w_d,$$
is surjective with all fibres of size $|W_1|^{d-1}$, and sends $\sg$-conjugacy
classes in $W$ to $\sg^d$-conjugacy classes in $W_1$.
\par
Now note that the above isomorphism $G^F\cong G_1^{F^d}$ sends $R_{T_w,1}$
to $R_{T_{p(w)},1}$.
Thus, it remains to show that $\tilde\rho(w\sg)=\tilde\rho_1(p(w)\sg^d)$ for
all $w\in W$. Clearly it is sufficient to check this for class representatives,
say on elements of the form $w=(w_1,1,\ldots,1)$, where $p(w)=w_1$.
The only terms contributing to the trace $\tilde\rho(w\sg)$ come from basis
elements $a_t$ ($t\in \bI$) with $w\sg.a_t=\pm a_t$, that is, with
$w\sg(t)w^{-1}=t$, whence $t\in C_W(w\sg)$.
Similarly, contributions to $\tilde\rho_1(w_1\sg^d)$ come from elements
$a_{t_1}$ with $t_1\in\bI_1\cap C_{W_1}(w_1\sg^d)$. Now
$$C_W(w\sg) =\{(g,\sg^d(g),\ldots,\sg^d(g))\mid g\in C_{W_1}(w_1\sg^d)\},$$
inducing a natural bijection
$$\bI\cap C_W(w\sg)\buildrel{1-1}\over\longrightarrow \bI_1\cap C_{W_1}(w_1\sg^d),
  \qquad  t=(t_1,\sg^d(t_1),\ldots,\sg^d(t_1))\mapsto t_1.$$
But it is clear from the definition of the involution module that
$$(w_1,1,\ldots,1)\sg.a_t=(w_1,1,\ldots,1)a_{(\sg^d(t_1),\ldots,\sg^d(t_1))}$$
and $w_1\sg^d.a_{t_1}=w_1.a_{\sg^d(t_1)}$ have the same sign.
\end{proof}

\subsection{}
Finally, as in \cite{Mar}, we now formally define Frobenius--Schur 
indicators for the combinatorially introduced unipotent characters of 
dihedral groups with non-trivial automorphism to obtain an analogue of 
Theorem~\ref{thm:main} in this case and show a unicity statement.

For this recall that there is a way to attach a set
$\Uch(\tw2I_2(m))$ of combinatorial objects, called ``unipotent characters'',
to the dihedral groups $I_2(m)$ ($m\ge3$) with non-trivial Coxeter
automorphism of order~2 (see \cite{exot}), such that each
$\chi\in\Uch(\tw2I_2(m))$ has a degree $\chi(1)\in\CC[q]$, and a Frobenius
eigenvalue $\Fr(\chi)$ (a root of unity). Following \cite{Mar} in the
untwisted case, we propose to introduce Frobenius-Schur indicators
$\nu(\chi)$ satisfying the following properties:
\begin{itemize}
\item[(1)] $\nu(\chi)\in\{0,\pm1\}$,
\item[(2)] $\nu(\chi)=0$ if and only if the Frobenius eigenvalue 
$\Fr(\chi)$ is non-real.
\end{itemize}
The set $\Uch(\tw2I_2(m))$ is subdivided into three families, two with one
element and one containing all the other characters; see \cite[6.1]{GM} where
one can also find the corresponding Fourier matrix. The characters in the
$1$-element families have Frobenius eigenvalue~$1$ and we set $\nu(\chi)=1$
for these. \par
The unipotent characters in the big family are parametrized by pairs $(k,l)$
of odd integers with $0<k<l<k+l<2m$, with corresponding Frobenius eigenvalue
given by $\zeta^{kl}$, where $\zeta$ is a $2m$th root of unity.
First assume that $m$ is even. Then $\zeta^{kl}$ is never real, so by~(2) above
we necessarily have $\nu(\chi)=0$ for all unipotent characters in this big
family. Now assume that $m$ is odd.
Then certainly $\zeta^{kl}$ is real when $l=m$. We propose to define the
Frobenius--Schur indicator $\nu(\chi)$ in this case to be $+1$ if $l=m$,
and~$0$ otherwise. With this, we have the following extension of
\cite[Thm.~1.2]{Mar}:

\begin{thm}   \label{thm:cox}
 Let $W=I_2(m)$ and $\sg$ the non-trivial Coxeter automorphism. Then the
 vector $(\nu(\chi))$ of Frobenius--Schur indicators of unipotent
 characters defined above is mapped under Fourier transform and restriction to
 the image of $\Irr(W)^\sg$ in $\Lambda(W,\sg)$ onto the vector of
 multiplicities in the decomposition of $\tilde{\rho}$ as a class function
 on $W.\sg$. \par
 Moreover, for any finite Coxeter group with non-trivial automorphism, the
 vector of Frobenius--Schur indicators is the only $\{1,-1, 0\}$-vector
 with this property and satisfying~(1), (2) above.
\end{thm}

\begin{proof}
For the 1-element families, the statement is trivially true, since the
corresponding Fourier matrix is the identity matrix and the multiplicities
in $\tilde{\rho}$ all equal~1 by Propositions~\ref{prop:I2modd}
and~\ref{prop:I2m}. When $m$ is odd, note that the data for $\tw2I_2(m)$ are
Ennola dual to those for $I_2(m)$, that is, the Fourier matrix is identical in
both cases, and the rows and columns are labelled by the same combinatorial
objects. The claim then follows by the computation in \cite[p.~27]{Mar}.
Finally, when $m$ is even, the first claim is obvious, since both sides are
the all $0$-vector. Unicity follows by a combinatorial argument as in
\cite[p.~27]{Mar}. \par
For the Weyl groups with non-trivial automorphism, the unicity is clear in
the case of 1-element families. For type $\tw2D_n$, uniqueness follows
exactly as in \cite{Mar}, and for the three exceptional types $\tw2E_6$,
$\tw3D_4$ and $\tw2F_4$, it is easily checked by computer.
\end{proof}


\end{document}